\documentclass{article}
\usepackage{amsmath}
\usepackage{amssymb}
\usepackage{latexsym}
\usepackage{amsthm}
\usepackage{mathrsfs} 
\usepackage{wasysym}
\usepackage{fancyhdr}
\usepackage{xcolor}
\usepackage{amsfonts}
\usepackage{amssymb}
\usepackage{enumitem}
\usepackage{epsfig}
\usepackage{epstopdf}
\usepackage[normalem]{ulem}
\usepackage[T1]{fontenc}
\usepackage{eucal}
\usepackage{tikz}
\usepackage[utf8]{inputenc}

\newtheorem{thm}{Theorem}[section]
\newtheorem{cor}[thm]{Corollary}
\newtheorem{lem}[thm]{Lemma}

\newtheorem{conj}[thm]{Conjecture}

\newtheorem{defn}[thm]{Definition}

\newcommand{\idq}{{}^{\alpha}_{\omega}}

\title{Packing Independent Cliques in $K_4$-minor-free Graphs 
}

\author{Benjamin Xiao\thanks{Central Magnet School, Murfreesboro, TN 37132.} \ and \  Dong Ye\thanks{Department of Mathematics, Middle Tennessee State University, Murfreesboro, TN 37132.} }

\date{}

\begin{document}
\maketitle

\begin{abstract}
Let $G$ be a graph and $S$ be a set of cliques of $G$. The set $S$ is an indeque set if every component of $G[S]$, the subgraph induced by vertices of $S$, is a clique. In this paper, we prove that the indeque ratio of $K_4$-minor-free graphs is $\frac 1 2$, which settle two conjectures of Biro, Collado and Zamora. 
We also show that the indeque ratio of subcubic graphs is $\frac 1 2$. 
\end{abstract}

\section{Introduction}
All graphs considered in this paper are simple. A graph is {\em complete} if any two vertices of the graph are adjacent. A complete graph on $n$-vertex is denoted by $K_n$. 
 Let $G$ be a graph. A {\em clique} of $G$ is a complete subgraph. Let $S$ be a set of cliques of $G$, and let $G[S]$ denote the subgraph of $G$ induced by all vertices of $S$. 
A set $S$ of cliques is {\em independent} if every component of $G[S]$ is a clique, which is also called an {\em indeque set}. The {\em indeque number} of $G$ is defined as the maximum size of indeque sets, denoted by
\[\idq(G)= \max\{|S| : S\subseteq V (G) \mbox{ and } G[S]\mbox{ is a union of independent cliques}\}.\]

An indeque set of $G$ is also called an {\em independent set} or {\em stable set} if every clique in $S$ is a single vertex (i.e. $K_1$). The {\em independent number} of a graph $G$ is the cardinality of the largest independent set of $G$, denoted by $\alpha(G)$. The {\em clique number} of a graph $G$ is the cardinality of the largest clique of $G$, denoted by $\omega(G)$. As noted in~\cite{BCZ}, it holds that 
\[\max\{\alpha(G), \omega(G)\}\le\, \idq(G)\le \alpha(G)\cdot \omega(G).\] 
It is well-known that it is NP-complete to determine the  independent number of a planar graph (cf. \cite{AHU}). From Four-Color Theorem~\cite{RT}, every $n$-vertex plane graph $G$ has an independent set of size at least $n/4$, which implies $\idq(G)\ge n/4$. 
For an $n$-vertex plane graph $G$, Biro, Collado and Zamora~\cite{BCZ} show $\idq(G) \ge 4n/15$ for an $n$-vertex planar graph, and further conjectured that $\idq(G)\ge 2n/5$ that is sharp if it holds. 

 An indeque set $S$ is called an {\em induced matching} if every clique is a single edge (i.e. $K_2$). The maximum size of an induced matching is denoted by $im(G)$. Hence it holds that $im(G)\le \idq(G)$.  Kanj et. al.~\cite{KPSX} proved that a 2-connected outerplanar graph $G$ with minimum degree two has a induced matching of size at least $n/2$, which implies that such an outerplanar graph $G$ has $\idq(G)\ge n/2$. It is NP-complete to determine $im(G)$ for planar graphs with maximum degree four~\cite{KS}, but the maximum induced matching problem can be solved in polynomial time for outerplanar graphs  (cf.~\cite{KK}). In general, it is hard to determine the indeque number of a graph even for planar graphs. For computational discussions, please refer to~\cite{ELWB}.

\begin{defn}
    Let $\mathcal G$ denote a class of graphs. Define $$\idq(\mathcal G, n) = \min\{\idq(G) : G \in \mathcal G \mbox { and } |G| = n\}.$$
    Then, the \textit{indeque ratio} of $\mathcal G$ is defined to be $$\idq(\mathcal G) = \lim_{n \to \infty}\frac{\idq(\mathcal G, n)}{n}.$$
\end{defn}

For a given class of graphs $\mathcal G$, in order to find $\idq(\mathcal G)$, it needs to find a sharp lower bound of indeque number for this class of graphs. Salia et. al.~\cite{SSTZ} obtained that $\idq(\mathcal G, n)=(n+o(n))/\log_2(n)$ for the class $\mathcal G$ of comparability graphs of $n$-element posets with acyclic cover graph, which implies $\idq(\mathcal G)=0$. The indeque ratio has been studied for the class of plane graphs by Biro, Collado and Zamora~\cite{BCZ}, who prove that the indeque ratio of the class of forests as stated below.

\begin{thm}[Biro, Collado and Zamora, \cite{BCZ}]\label{thm:forest}
    Let $\mathcal F$ be the class of forests. Then 
    \[\idq(\mathcal F)=\frac 2 3.\]
\end{thm}

Biro, Collado and Zamora~\cite{BCZ} also proved the indeque ration of the class of graphs with path-width at most two is $1/2$, and they further made the following conjectures about the indeque ratio of the classes of series-parallel and outerplanar graphs. 

\begin{conj}[Biro, Collado and Zamora, \cite{BCZ}]\label{conj1}
    Let $\mathcal S$ be the class of series-parallel graphs. Then 
    \[\idq(\mathcal S)=\frac 1 2.\]
\end{conj}

\begin{conj}[Biro, Collado and Zamora, \cite{BCZ}]\label{conj2}
    Let $\mathcal O$ be the class of outerplanar graphs. Then
    \[\idq(\mathcal O)=\frac 1 2.\]
\end{conj}

As pointed out in~\cite{BCZ}, the indeque ratio of outerplanar graphs is at least $4/9$ which follows from the above theorem and a result of Hosono on maximum induced forests of outerplanar graphs~\cite{Ho}.  It follows directly from induced matching result of Kanj et. al.~\cite{KPSX} on outerplanar graphs that the indeque ratio for 2-connected outerplanar graphs with a minimum degree at least two is 1/2.

In this paper, we study the indeque ratio of $K_4$-minor free graphs. It is well-known that the $K_4$-minor-free graphs form a subclass of plane graphs. Along with this, Wald and Colbourn~\cite{WC} proved that $K_4$-minor-free graphs are equivalent to graphs with tree-width at most $2$. Hence the class of $K_4$-minor free graphs also include graphs with path-width at most two as a subclass. Both series-parallel graphs and outerplanar graphs are $K_4$-minor-free graphs. Every block of a $K_4$-minor-free graph is a series-parallel graph. The following is our main result.

\begin{thm}\label{thm:main}
Let $\mathcal K$ be the class of $K_4$-minor free graphs. Then
\[\idq(\mathcal K)=\frac 1 2 .\]
\end{thm}

Since $\mathcal O\subset \mathcal K$ and $\mathcal S\subset \mathcal K$, it follows that $\idq(\mathcal K)\le\idq(\mathcal O)$ and $\idq(\mathcal K)\le \idq(\mathcal S)$. Note that the $4$-cycle $C_4$ is a series-parallel and outerplanar graph which has $\idq(C_4)=2$. The union of disjoint $C_4$s shows that $\idq(\mathcal O)\le 1/2$ and $\idq(\mathcal S)\le 1/2$. Hence, Theorem~\ref{thm:main} settles Conjectures~\ref{conj1} and~\ref{conj2}. 
In the last section, we show that the class of subcubic graphs also has indeque ratio $1/2$. 


\section{$K_4$-minor-free graphs}

A graph is $K_4$-minor-free if it does not contain $K_4$ as a minor. It is well known that the $K_4$-minor-free graph contains two important subfamily of graphs: one is the family of outer-planar graphs and the other one is the family of series-parallel graphs. Duffin~\cite{D} revealed the connection between the $K_4$-minor-free graphs and the series-parallel graphs as follows.

\begin{thm}[Duffin~\cite{D}]\label{thm:parallel}
    A graph without cut-vertex is $K_4$-minor-free graph if and only if it is a series-parallel graph. 
\end{thm}

Hence, every block, defined as a maximal 2-connected subgraph, of a $K_4$-minor free graph is a series-parallel graph. The definition of a series-parallel graph is given as follows. 
A {\em 2-terminal graph} $G$ is a graph with two special vertices, called {\em source vertex $s$} and {\em sink vertex $t$}, denoted by $(G,s,t)$. A {\em series composition} of two 2-terminal graphs $(G_1,s_1,t_1)$ and $(G_2,s_2,t_2)$ is a new 2-terminal graph formed by $G_1$ and $G_2$ by identifying $t_1=s_2$ with source vertex $s_1$ and sink vertex $t_2$, denoted by $(G_1\bullet G_2, s_1, t_2)$. A {\em parallel composition} of two 2-terminal graphs $(G_1,s_1,t_1)$ and $(G_2,s_2,t_2)$ is a new 2-terminal graph formed by $(G_1, s_1,t_1)$ and $(G_2, s_2, t_2)$ by identifying $s_1$ and $s_2$ as the new source vertex $s$ and identifying $t_1$ and $t_2$ as the new sink vertex $t$, denoted by $(G_1\|G_2, s, t)$. A  graph is called a {\em series-parallel graph} if it can be constructed from $K_2$ by series and parallel compositions. A series-parallel graph can be recognized in linear time \cite{VTL}.

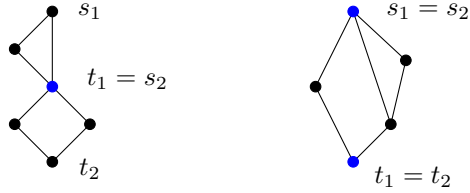
\begin{figure}[!htb]
\begin{center}
\begin{tikzpicture}
    \draw [](0,1) -- (-.5, .5)--(0,0)--(0,1);
    \draw [](0,0) -- (-.5,-.5)--(0,-1)--(.5, -.5)--(0,0);
    \filldraw[blue](0,0) circle (2pt);
    \filldraw[](0,1) circle (2pt);
    \filldraw[](0,-1) circle (2pt);
    \filldraw[](-.5,-.5) circle (2pt);
    \filldraw[] (.5,-.5) circle (2pt);
    \filldraw[] (-.5,.5) circle (2pt);
    \node at (1,0.1) {$t_1=s_2$};
    \node at (.5,1) {$s_1$ };
    \node at (.5,-1.1) {$t_2$ };

    \draw[](4,1) -- (3.5, 0) -- (4,-1);
    \draw[](4.5, -0.5)-- (4,1) -- (4.7, 0.35) -- (4.5, -0.5) -- (4,-1);
    \filldraw[blue](4,1) circle (2pt);
    \filldraw[](3.5,0) circle (2pt);
    \filldraw[blue](4,-1) circle (2pt);
    \filldraw[](4.7,0.35) circle (2pt);
    \filldraw[](4.5,-0.5) circle (2pt);

    \node at (5,1) {$s_1=s_2$};
    \node at (4.8, -1.2) {$t_1=t_2$};
\end{tikzpicture}
\caption{A series composition (Left) and a parallel composition (Right).}\label{fig1}
\end{center}
\end{figure}   

The following is the main result that we are going to prove.

\begin{thm}\label{thm:lower}
    Let $G$ be a $K_4$-minor-free graph. Then
    \[\idq(G)\ge  \frac 1 2 |V(G)|.\]
\end{thm}

Before proceeding to prove Theorem~\ref{thm:lower}, we need some terminologies that will be used throughout the proof. Let $G$ be a $K_4$-minor-free graph.  
A {\em leaf-block} $H$ is block of $G$ sharing at most one vertex $v$ with other blocks of $G$, and the vertex $v$ is a cut-vertex separating $H-v$ and $G-H$ if it exists. 

Consider a leaf-block $H$ of a $K_4$-minor-free graph $G$. By Theorem~\ref{thm:parallel}, $H$ is a series-parallel graph. A maximal subgraph of $H$ constructed from $K_2$s by only series composition is called a {\em $0$-series-piece} of $H$. A $0$-series-piece of $H$ is a path of which every internal vertex has degree two in $G$. A maximal subgraph of $H$ with two terminals constructed from $0$-series-pieces by only parallel compositions is called a {\em $0$-parallel-piece}. A {\em $k$-series-piece} is a maximal subgraph constructed from $i$-parallel-pieces with $0\le i< k$ and $j$-series-pieces with $0\le j<k$ by only series compositions where at least one of these pieces is a $(k-1)$-parallel-piece. Then, a {\em $k$-parallel-piece} is a maximal subgraph constructed from $i$-series-pieces with $0\le i\le k$ and $j$-parallel-piece with $0\le j<k$ by only parallel compositions where at least one of these pieces is a $k$-series-piece. Note that a parallel-piece is 2-connected. For convenience, a $(k-1)$-parallel-piece is considered as a trivial $k$-series-piece, but a $k$-series-piece is not considered as a $k$-parallel-piece. A vertex of a $k$-series-piece or parallel-piece other than the two terminal vertices is called an {\em internal} vertex. 

\begin{lem}[Eppstein, \cite{DE}]\label{lem:terminal}
Let $H$ be a 2-connected series-parallel graph. Then, for any edge $st$, $H$ is a 2-terminal graph with $s$ and $t$ as terminal vertices.
\end{lem}

For any $0$-series-piece $P$ of $H$ with terminal vertices $s$ and $t$, let $H'$ be a graph obtained from $H$ by replacing $P$ by an edge $st$. Then $H'$ is a 2-connected series-parallel graph. By Lemma~\ref{lem:terminal}, $H'$ is a 2-terminal graph with $s$ and $t$ as terminal vertices. Then $H$ is obtained from $H'-e$ and $P$ by a parallel composition. So $H$ is a 2-terminal graph with $s$ and $t$ as terminal vertices. So we have the following corollary.

\begin{cor}\label{cor:terminal}
Let $H$ be a 2-connected series-parallel graph. Then for any $0$-series-piece $P$ with terminal vertices $s$ and $t$, $H$ is a $k$-parallel-piece with terminal vertices $s$ and $t$ for some integer $k$. 
\end{cor}

In order to prove Theorem~\ref{thm:lower}, we consider a minimum counterexample $G$ with a leaf-block $H$. If $H$ has a subgraph $X$ that has an indeque set $S$ such that  $|S|\ge \frac 1 2 |V(X)|$ and every vertex of $S$ is not adjacent to any vertex in $G-X$, then any indeque set of $G-X$ together with $S$ form an indeque set of $G$. Note that $G-X$ is $K_4$-minor-free and has fewer vertices than $G$. As $G$ is a minimum counterexample, $G-X$ is not a counterexample and hence it has an indeque set $S'$ of size at least $\frac 1 2 |V(G-X)|$. So $S\cup S'$ is an indeque set of $G$. Therefore,
\[\idq(G)\ge |S\cup S'|=|S|+|S'|\ge \frac 1 2 |V(X)|+\frac 1 2 (|V(G)|-|V(X)|)=\frac 1 2 |V(G)|,\]
which leads a contradiction to that $G$ is a counterexample. Such a pair $(X, S)$ is called a {\em contra-pair} in the following. Hence, it suffices to show that $G$ always has a contra-pair.

\begin{figure}[!htb]
\begin{center}
\begin{tikzpicture}
    \draw [](0,1) -- (.5, .5)--(0,0)--(0,1); \filldraw[] (.5,.5) circle (2pt);
  \node at (0,1.3) {$s=v$ };
    \node at (0,-.3) {$t$ };
\node at (.3, -1) {$\Gamma_1$};

    \draw [](2,1) -- (2,0)--(2.5, .5)--(2,1);
    \filldraw[ ](0,0) circle (2pt);
    \filldraw[blue](0,1) circle (2pt);
    \filldraw[](2,0) circle (2pt);
    \filldraw[blue] (2.5,.5) circle (2pt);
    \filldraw[] (2,1) circle (2pt);
    \node at (2.8,0.5) {$v$};
    \node at (2,1.3) {$s$ };
    \node at (2,-.3) {$t$ };
    \node at (2, -1) {$\Gamma_2$};

    \begin{scope}[shift={(2cm,0)}]
      \draw [](2,1) -- (1.5,.5)--(2,0)--(2.5, .5)--(2,1);
      \draw[dashed] (2,1)--(2,0);
    \filldraw[ ](0,0) circle (2pt);
    \filldraw[](0,1) circle (2pt);
    \filldraw[](2,0) circle (2pt);
    \filldraw[](1.5,.5) circle (2pt);
    \filldraw[blue] (2.5,.5) circle (2pt);
    \filldraw[] (2,1) circle (2pt);
    \node at (2.8,0.5) {$v$};\node at (2,1.3) {$s$ };
    \node at (2,-.3) {$t$ };
\node at (2, -1) {$\Gamma_3$};
    \end{scope}

    \begin{scope}[shift={(4cm,0)}]
      \draw [](2,1) --(2,0)--(2.5, .2)--(2.5,.8)--(2,1) ;
    \filldraw[ ](0,0) circle (2pt);
    \filldraw[](0,1) circle (2pt);
    \filldraw[](2,0) circle (2pt);
     
    \filldraw[blue] (2.5,.8) circle (2pt);
    \filldraw[] (2.5,.2) circle (2pt);
    \filldraw[] (2,1) circle (2pt);
    \node at (2.8,0.8) {$v$};\node at (2,1.3) {$s$ };
    \node at (2,-.3) {$t$ };
\node at (2.3, -1) {$\Gamma_4$};
    \end{scope}

     \begin{scope}[shift={(6cm,0)}]
      \draw [](2,1) -- (1.5,.5)--(2,0)--(2.5, .2)--(2.5,.8)--(2,1);
      \draw[dashed] (2,1)--(2,0);
    \filldraw[ ](0,0) circle (2pt);
    \filldraw[](0,1) circle (2pt);
    \filldraw[](2,0) circle (2pt);
    \filldraw[](1.5,.5) circle (2pt);
    \filldraw[blue] (2.5,.8) circle (2pt);
    \filldraw[] (2.5,.2) circle (2pt);
    \filldraw[] (2,1) circle (2pt);
    \node at (2.8,0.8) {$v$};\node at (2,1.3) {$s$ };
    \node at (2,-.3) {$t$ };
\node at (2.3, -1) {$\Gamma_5$};
    \end{scope}
    


\begin{scope}[shift={(8cm,0)}]
      \draw [](2,1) --(2,0)--(2.5, .1)--(2.8,.5)--(2.5,.9)--(2,1) ;
      \filldraw[ ](0,0) circle (2pt);
    \filldraw[](0,1) circle (2pt);
    \filldraw[](2,0) circle (2pt);
     \filldraw[] (2.5,.9) circle (2pt);
    \filldraw[blue] (2.8,.5) circle (2pt);
    \filldraw[] (2.5,.1) circle (2pt);
    \filldraw[] (2,1) circle (2pt);
    \node at (3.1,0.5) {$v$};\node at (2,1.3) {$s$ };
    \node at (2,-.3) {$t$ };
\node at (2.3, -1) {$\Gamma_6$};
    \end{scope} 
    
\end{tikzpicture}
\caption{All possible 0-parallel-pieces with $v$ (edges in dashed lines may or may not be present).}\label{fig:0-parallel-piece}
\end{center}
\end{figure}
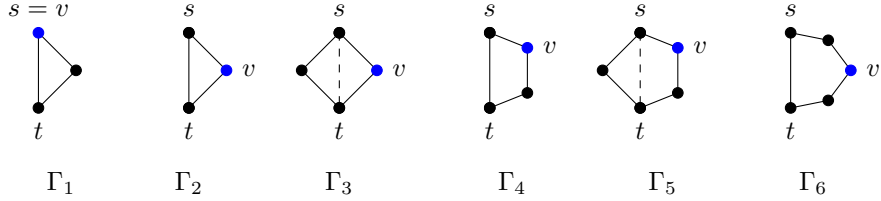

\begin{lem}\label{lem:0-piece}
Let $G$ be a minimum counterexample to Theorem~\ref{thm:lower}, and let $H$ be a leaf-block of $G$ with a cut-vertex $v$. Then:

{\upshape (1)} $G$ does not contain two adjacent vertices of degree two, and hence every $0$-series-piece of $H$ is a path with at most three vertices if it does not contain $v$ as an internal vertex.

{\upshape (2)} every $0$-parallel-piece of $H$ is either a triangle or one of the six graphs containing $v$ in Figure~\ref{fig:0-parallel-piece} up to the symmetry exchanging $s$ and $t$. 
\end{lem}

\begin{proof}
Let $G$ be a $K_4$-minor-free graph with $\idq(G)<\frac 1 2 |V(G)|$ such that $|V(G)|$ is minimum, and let $H$ be a leaf-block with a cut-vertex $v$. \medskip 

(1) Suppose to the contrary that $G$ has two adjacent degree-2 vertices $v_1$ and $v_2$ different from $v$, and let $P:=v_3v_1v_2v_4$ (which could be a triangle). So $|V(P)|\le 4$. Then let $S=\{v_1,v_2\}$. So $S$ is an indeque set of $P$ and both $v_1$ and $v_2$ are not adjacent to any vertex of $G-P$. Hence $(P,S)$ is a desired contra-pair, a contradiction. So (1) holds. \medskip

(2) By Theorem~\ref{thm:parallel}, $H$ is a series-parallel graph. Note that all vertices of $H-v$ are not adjacent to any vertex of $G-H$.

Let $Q$ be a $0$-parallel-piece of $H$ with two terminals $s$ and $t$. Let $P$ be a $0$-series-piece in $Q$. Then every internal vertex of $P$ except $v$ has degree two in $G$. If $P$ does not contain $v$ as an internal vertex, then it follows from (1) that $|V(P)|\le 3$. If $P$ does contain $v$ as an internal vertex, then $|V(P)|\le 5$ and $|V(P)|=5$ if and only if $v$ is the middle vertex of $P$ by (1). So $Q-\{s,t, v\}$ is an independent set of $G$, denoted by $S$. So $(Q, S)$ is a contra-pair if $|S|\ge \frac 1 2 |V(Q)|$, which yields a contradiction. So $|S|<\frac 1 2 |V(Q)|$. 

If $v\in \{s,t\}$, then $|S|=|V(Q)|-2$. Therefore, $|V(Q)|-2=|S|<\frac 1 2 |V(Q)|$ which implies that $|V(Q)|\le 3$. Hence $Q$ is $\Gamma_1$. 

If $v\notin \{s,t\}$, then $|S|=|V(Q)|-3$. Therefore, $|V(Q)|-3<\frac 1 2 |V(Q)|$ which implies that $|V(Q)|\le 5$. If $v$ is not adjacent to $s$ or $t$, then $v$ is the middle vertex of a $0$-series-piece and $Q$ is $\Gamma_6$ in Figure~\ref{fig:0-parallel-piece} . If $v$ is adjacent to exactly one of $s$ and $t$, without loss of generality, assume $v$ is adjacent to $s$. It follows from (1) that $Q$ is either $\Gamma_4$ or $\Gamma_5$ in Figure~\ref{fig:0-parallel-piece}.  If $v$ is adjacent to both $s$ and $t$, then $v$ is not adjacent to any vertex of $S$. So consider $(Q-v, S)$ which is a contra-pair if $|S|\ge \frac 1 2 (|V(Q)|-1)$. Hence $|V(Q)|-3=|S|<\frac 1 2 (|V(Q)|-1)$ which implies that $|V(Q)|\le 4$. Therefore, $Q$ is either $\Gamma_2$ or $ \Gamma_3 $ in Figure~\ref{fig:0-parallel-piece}. So (2) follows. This completes the proof of the lemma. 
\end{proof}

A {\em triangle-string} is a series-parallel graph that can be constructed from a triangle-string by recursively performing one of the following two operations: (1) a parallel composition with a $K_2$, and (2) a series composition with a triangle.  
A {\em triangle-ring} is obtained from two triangle-strings by a parallel composition. 

Note that all $0$-parallel-pieces of  a  triangle-string or a triangle-ring are triangles. If one of the triangle of a triangle-string or a riangle-ring is replaced by some $\Gamma_i$ in Figure~\ref{fig:0-parallel-piece}, the resulting series-parallel subgraph is called a $\Gamma_i$-triangle-string or a $\Gamma_i$-triangle-ring.

\begin{figure}[!htb]
\begin{center}
    \begin{tikzpicture}
    \draw [](-2,0)--(-1.5,0.5)--(-1,0)--(-2,0);
    \draw [](-1,0)--(-0.5,0.5)--(0,0)--(-1,0);
    \draw [](0,0)--(0.5,0.5)--(1,0)--(0,0);
    \draw [](1,0)--(1.5,0.5)--(2,0)--(1,0);
    \draw [](2,0)--(2.5,0.5)--(3,0)--(2,0);
    \draw [](0,0) arc(-60:-120:2);
    \draw [](2,0) arc(-60:-120:4);
    \filldraw[](-2,0) circle (2pt);
    \filldraw[](-1,0) circle (2pt);
    \filldraw[](0,0) circle (2pt);
    \filldraw[](1,0) circle (2pt);
    \filldraw[](2,0) circle (2pt);
    \filldraw[](3,0) circle (2pt);
    \filldraw[](-1.5,0.5) circle (2pt);
    \filldraw[](-0.5,0.5) circle (2pt);
    \filldraw[](0.5,0.5) circle (2pt);
    \filldraw[](1.5,0.5) circle (2pt);
    \filldraw[](2.5,0.5) circle (2pt);

\begin{scope}[shift={(8,0)}]

   \draw [](-2,0)--(-1.5,0.5)--(-1,0)--(-2,0)--(-3,0);
    \draw [](-1,0)--(-0.5,0.5)--(0,0)--(-1,0);
    \draw [](0,0)--(0.5,0.5)--(1,0)--(0,0);
    \draw [](1,0)--(1.5,0.5)--(2,0)--(1,0);
    \draw [](2,0)--(2.5,0.5)--(3,0)--(2,0);
    \draw [](0,0) arc(-60:-120:2);
    \draw [](2,0) arc(-60:-120:4);
    \filldraw[](-2,0) circle (2pt);
    \filldraw[](-1,0) circle (2pt);
    \filldraw[](0,0) circle (2pt);
    \filldraw[](1,0) circle (2pt);
    \filldraw[](2,0) circle (2pt);
    \filldraw[](3,0) circle (2pt);
    \filldraw[](-3,0) circle (2pt);
    \filldraw[](-1.5,0.5) circle (2pt);
    \filldraw[](-0.5,0.5) circle (2pt);
    \filldraw[](0.5,0.5) circle (2pt);
    \filldraw[](1.5,0.5) circle (2pt);
    \filldraw[](2.5,0.5) circle (2pt);
 
    \node at (-3.5,0) {$s$};
    \node at (3.5,0) {$t$ };
   \node at (-2,.3) {$x$ };

     
   

\end{scope}
   
\end{tikzpicture}
\caption{A triangle-string (Left) and a kite (Right).}\label{fig2}
\end{center}
\end{figure}
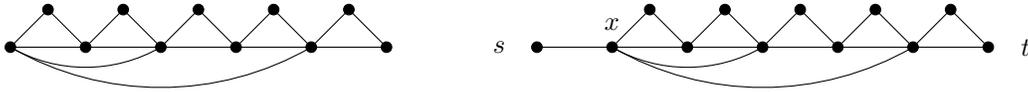

A {\em kite}  (or {\em $\Gamma_i$-kite}) is a series-parallel graph constructed from a $0$-series-piece $P$ (that is a path) and a  triangle-string ( or a  $\Gamma_i$-triangle-string). The terminal vertex of $P$ identified with a terminal vertex of the triangle-string (resp. $\Gamma_i$-triangle-string) is called the {\em joint vertex} of the kite (or $\Gamma_i$-kite). For example, see Figure~\ref{fig2} and $x$ is the joint vertex. 
  
In the following, we are going to prove Theorem~\ref{thm:lower}.

\medskip
\noindent{\em \bf Proof of Theorem~\ref{thm:lower}.} 
    Suppose to the contrary that the statement does not hold and let $G$ be a counterexample with the minimum number of vertices. In the following, it suffices to prove that $G$ has a contra-pair, which leads to a contradiction.

If $G$ has a vertex $u$ of degree one, let $w$ be the unique neighbor of $u$. Then let $X=uv$ and $S=\{u\}$, and $(X, S)$ is a desired contra-pair.
 So assume that $G$ has no vertex of degree one. It follows that a vertex of $G$ is either a cut-vertex or belongs to a block. Let $H$ be a leaf-block of $G$.

\medskip
\noindent{\bf Claim~1.} {\sl $H$ does not contain a triangle-ring or a $\Gamma_1$-triangle-ring.}
\medskip

\noindent{\em Proof of Claim~1.} If not, let $R$ be a triangle-ring or a $\Gamma_1$-triangle-ring. Let $S$ be the set of all degree-2 vertices of $R$. Then  $S$ is an independent set of $G$ and $|S|=\frac 1 2 |V(R)|$. Hence $(R, S)$ is a desired contra-pair, a contradiction. 
So the claim holds. 
\medskip

According to Claim~1, a parallel composition cannot be performed on two triangle-strings that do not include $v$ as an internal degree-2 vertex. By a similar argument and (1) of Lemma~\ref{lem:0-piece}, a parallel composition also cannot be applied to a triangle-string and a $0$-series-piece that does not include $v$ as an internal degree-2 vertex.

\medskip
\noindent{\bf Claim~2.} {\sl $H$ does not contain a kite that does not contain $v$ as an internal degree-2 vertex or its joint vertex.}
\medskip

\noindent{\em Proof of Claim~2.} If not, let $K$ be such a kite with $x$ as its joint vertex. If $K$ contains $v$, then $v$ is either the terminal vertex of $K$ or internal vertex of degree at least three. So $K$ may contain $\Gamma_1$ (or $\Gamma_1$-kite). Assume $K$ is formed by a $0$-series-piece $P$ and a triangle-string or $\Gamma_1$-triangle-string. Without loss of generality, assume $s$ is the terminal vertex of $K$ in $P$. By (1) of Lemma~\ref{lem:0-piece},  $P$ has at most three vertices. Let $K'=K-s$ if $P$ has three vertices and $K'=K$ if $P$ has exactly two vertices. 
Let $S:=\{\mbox{all degree-2 vertices of } K' \mbox{ except terminal vertices} \}\cup \{x\}$. Then $|S|=\frac 1 2 |V(K')|$, and all vertices of $S$ have no neighbors in $G-K'$. Therefore, $(K', S)$ is a desired contra-pair, a contradiction. 
Hence Claim~2 holds. 
\medskip

By Claim~2, a series composition cannot be performed on a triangle-string and a $0$-series-piece.  Let $H'$ be a $k$-series-piece or a $k$-parallel-piece not containing $v$ as an internal vertex. It follows from Lemma~\ref{lem:0-piece} that every $0$-parallel-piece of $H'$ is a triangle. By Claim~2, a series composition of $H'$ can only be applied to two triangle-strings, and by Claim~1, a parallel composition of $H'$ can only be applied to a $K_2$ and a triangle-string. Therefore, the following holds, \medskip
\begin{description}
\item ($*$) {\sl every $k$-series piece with $k\ge 1$ and every $k$-parallel piece with $k\ge 0$ not containing $v$ as an internal vertex is a triangle-string. }
 \end{description}
\medskip

First, assume that $v$ is contained in a $0$-parallel-piece $Q$ with terminal vertices $s$ and $t$. Then it follows from (2) of Lemma~\ref{lem:0-piece} that  $Q\in \{\Gamma_i | i\in [6]\}$. 
By Corollary~\ref{cor:terminal}, $H$ is a 2-terminal graph with terminal vertices $s$ and $t$. Then either $H=Q$ or $H$ is obtained from $Q$ and $H'$ by a parallel composition where $H':=H-(Q-\{s,t\})$.

If $H=Q$, then $H=\Gamma_i$ for some $i$. It is easy to check that $H$ has an indeque set $S\subseteq V(H-v)$ with size $|S|\ge \frac 1 2 |V(H)|$. Hence $(H,S)$ is a desired contra-pair, a contradiction. So in the following, assume that $H\ne Q$ and hence it is obtained from $Q$ and $H'$ by a parallel composition. Note that $H'$ is not a $0$-series-piece because $H\ne Q$ (by the maximality of $0$-parallel-piece). So $H'$ is either a $k$-parallel-piece with $k\ge 0$ or a $k$-series-piece with $k\ge 1$. Since $v\notin H'$, it follows from ($*$) that $H'$ is a triangle-string. Therefore, $H$ is a $\Gamma_i$-triangle-ring. By Claim~1, $i\ne 1$.

If $H$ is a $\Gamma_2$-triangle-ring, let the two neighbors of $v$ in $H$ be $v_1$ and $v_2$. Without loss of generality  assume $v_2$ is not a terminal vertex of $H$. Let $S=\{\mbox{all degree-2 vertices of } H\}\cup \{v_2\}$. Then $S$ is an indeque set of $H$ which are not adjacent to any vertices of $G-H$, and $|S|= \frac 1 2 |V(H)|$. So $(H,S)$ is a desired contra-pair. 

If $H$ is a $\Gamma_3$-triangle-ring. Let $S=\{\mbox{all degree-2 vertices of }H\}-\{v\}$ and then $ |S|= \frac 1 2 (|V(H)|-1)$.  Then $S$ is an independent set of $G$. So $(H-v, S)$ is a desired contra-pair. 

If $H$ is a $\Gamma_4$-triangle-ring, let $v_1$ be the neighbor of $v$ in $H$ that is a terminal vertex of $\Gamma_4$. Then all neighbors of $v_1$ belong to $H$. Let $S=\{\mbox{all degree-2 vertices of } H\}\cup \{v_1\}-\{v\}$, and $|S|= \frac 1 2 (|V(H)|+1)\ge \frac 1 2 |V(H)|$. Hence $(H,S)$ is a desired contra-pair. 

Now assume that $H$ is a $\Gamma_i$-triangle-ring for $i\in \{5, 6 \}$,  let 
$S=  \{\mbox{ all degree-2 vertices of } H \}-\{v\} $. Then $S$ is an independent set of $G$ and has size $|S|\ge \frac 1 2 |V(H)|$. Therefore, $(H,S)$ is a desired contra-pair, a contradiction. Hence $i\notin \{5,6\}$.  This contradiction implies that $v$ does not belong to a $0$-parallel-piece.  \medskip

So the smallest integer $k$ that a $k$-parallel-piece containing $v$ is at least 1. Assume that $v$ is contained in a $0$-series-piece $P$ with terminal vertices $s$ and $t$. Then $P$ has at most five vertices by (1) of Lemma~\ref{lem:0-piece}. By Corollary~\ref{cor:terminal}, $H$ is constructed from $P$ and a 2-terminal graph $H'$. Since $v$ is not contained in a $0$-parallel-piece, $H'$ is not a $0$-series-piece. 
It follows from ($*$) that $H'$ is a triangle-string. If $P$ has at least three vertices, then $H$ contains a kite or $\Gamma_1$-kite as described in Claim~2, a contradiction. So $P$ has exactly two vertices $s$ and $t$ and one of them is $v$.  Without loss of generality, assume $v=s$. 
Let $S=\{\mbox{all degree-2 vertices of }H\}\cup \{t\}$. Then $S$ is an indeque set of $G$ and every vertex of $S$ is not adjacent to a vertex of $G-H$. Note that $|S|\ge \frac 1 2 |V(H)|$. Hence $(H, S)$ is a desired contra-pair. This completes the proof of Theorem~\ref{thm:lower}
\qed


\medskip 

It follows from Theorem~\ref{thm:lower} that the indeque ratio of the class of $K_4$-minor-free graphs satisfies $\idq(\mathcal K) \ge \frac{1}{2}$. On the other hand, there are infinitely many $K_4$-minor-free graphs having indeque ratio $1/2$: for example, the 4-cycle and the triangle-rings etc. 
So $\idq(\mathcal K)\le \frac 1 2$. Therefore, $\idq(\mathcal K)=\frac 1 2$, and Theorem~\ref{thm:main} holds.

\section{Concluding Remarks}

A {\em subcubic graph} is a graph with maximum degree at most three. Hsieh, Le, Le and Peng~\cite{HLLP} recently show that it is NP-complete to determine the indeque number of subcubic planar graphs (Cluster Vertex-Deletion or 2-Claw Vertex-Deletion Problem, see Theorem~2 in \cite{HLLP}). It is also NP-complete to determine the inducing matching number of subcubic planar graphs~\cite{HOV}. In~\cite{KMM}, Kang, Mnich and M\"uller provided a polynomial-time algorithm to find an induced matching of size at least $m/9$ in a subcubic planar graph with $m$ edges.  

The following result gives a sharp lower bound for indeque number of subcubic graphs, and consequently  the indeque ratio of  subcubic graphs.

\begin{thm}
Let $G$ be an $n$-vertex subcubic graph. Then 
\[\idq(G)\ge \frac n 2,\]
and the bound is sharp. 
\end{thm}
\begin{proof}
    Let $G$ be a subcubic graph. Let $(X,Y)$ be a vertex partition of $G$ that maximizes the size of $E(X,Y)$, the set of edges joining a vertex of $X$ and a vertex of $Y$. If one part of $(X,Y)$, say $X$, has a vertex $v$ with degree at least two neighbors in its own part $X$. Then the partition $(X\backslash \{v\}, Y\cup \{v\})$ is a new vertex partition and 
    \[|E(X\backslash \{v\}, Y\cup \{v\})|>|E(X,Y)|,\] which contradicts that the maximality of $|E(X,Y)|$. Therefore, every vertex of $X$ or $Y$ has at most one neighbor in its own part, which means the vertex-induced subgraphs $G[X]$ and $G[Y]$ have maximum degree at most one. 
    Hence both $X$ and $Y$ are indeque sets of $G$. Note that $|X|+|Y|=|V(G)|=n$. It follows 
    \[\idq(G)\ge \max\{|X|, |Y|\}\ge \frac n 2.\]
    This completes the proof of the first part.
    
    To see the sharpness, there are 2-connected subcubic graphs with $4k$ vertices having indeque number $2k$: consider the subcubic graphs constructed from a cycle by blowing up every vertex to a square. It is quite easy to see that the graph has an indeque set of size $2k$. The graph does not have an indeque set of size at least $2k+1$. Otherwise, one of the 4-cycles contains at least three vertices from the indeque set, which is impossible because a 4-cycle has an indeque number two. 
\end{proof}

It follows immediately that the indeque ratio of all subcubic graphs is $1/2$. 

\begin{thm}
Let $\mathcal C$ be the class of all subcubic graphs. Then the indeque ratio of $\mathcal C$ is 
\[\idq(\mathcal C)=\frac 1 2.\]
\end{thm}

Note that the 3-dimensional hypercube $Q_3$ has indeque number 4. As noted by Biro, Collado and Zamora~\cite{BCZ}, the indeque ratio of any class $\mathcal G$ of graphs is closed under disjoint union.  Hence the indeque ratio for cubic graphs is also $1/2$.  

It was conjectured by Biro, Collado and Zamora~\cite{BCZ} that the indeque ratio of plane graphs is $2/5$. Note that, $K_4$-minor-free graphs, subcubic graphs and plane graphs have small maximum average degree. It is interesting to study the family of graphs with bounded maximum average degree. 



\end{document}